\documentclass[12pt]{amsart}

\usepackage[margin=2.4cm]{geometry}
\usepackage{a4wide}

\usepackage{amssymb}
\usepackage{enumerate}
\usepackage{mathrsfs}
\usepackage{stackrel}

\newtheorem{theorem}{Theorem}
\newtheorem{corollary}[theorem]{Corollary}
\newtheorem{lemma}[theorem]{Lemma}
\newtheorem{proposition}[theorem]{Proposition}
\theoremstyle{remark}
\newtheorem{remark}[theorem]{Remark}

\newcommand{\lur}{\mathsf{LUR}}
\newcommand{\norm}[1]{\left\|#1\right\|}
\makeatletter
\def\@tvsp{\mathchoice{{}\mkern-4.5mu}{{}\mkern-4.5mu}{{}\mkern-2.5mu}{}}
\def\ltrivert{\left|\@tvsp\left|\@tvsp\left|}
\def\rtrivert{\right|\@tvsp\right|\@tvsp\right|}
\makeatother
\newcommand{\tnorm}[1]{\ltrivert #1 \rtrivert}
\newcommand{\normi}[1]{\norm{ #1 }_{\infty}}
\newcommand{\normt}[1]{\norm{ #1 }_2}

\begin{document}

\title[On wLUR norms which are not LUR]{On weakly locally uniformly rotund norms which are not locally uniformly rotund}
\author[S. Draga]{Szymon Draga}
\email{szymon.draga@gmail.com}
\address{Institute of Mathematics, University of Silesia, Bankowa 14, 40-007 Katowice, Poland}
\email{szymon.draga@gmail.com}
\date{}

\begin{abstract}
We show that every infinite-dimensional Banach space with separable dual admits an equivalent norm which is weakly locally uniformly rotund but not locally uniformly rotund.
\end{abstract}


\maketitle

\section{Introduction}
Recall that a norm in a Banach space is called \textit{strictly convex} ($\mathsf{SC}$) if for arbitrary points $x$, $y$ from the unit sphere the equality $\norm{x+y}=2$ implies that ${x=y}$. The norm is called \textit{weakly locally uniformly rotund} ($w\lur$) if for arbitrary points $x_n$ ($n=1,2,\ldots$) and $x$ from the unit sphere the equality $\lim_{n\to\infty}\norm{x_n+x}=2$ implies the weak convergence of the sequence $(x_n)_{n=1}^\infty$ to $x$; if the last convergence is strong, then the norm is called \textit{locally uniformly rotund} ($\lur$). In the preceding definitions it is sufficient to require that $\lim_{n\to\infty}\norm{x_n}=\norm{x}$ and $\lim_{n\to\infty}\norm{x_n+x}=2\norm{x}$.

It is clear that $w\lur\implies\mathsf{SC}$ and $\lur\implies w\lur$; it is also well-known that none of these implications reverses. Indeed, the space $\ell_\infty$ can be renormed in a strictly convex manner, but it does not admit an equivalent $w\lur$ norm (cf. \cite[\S 4.5]{diestel1975}). M.A. Smith \cite[Example 2]{smith1978} gave an example of a $w\lur$ norm on $\ell_2$ which is not $\lur$, however, in the next section we shall present a somewhat simpler example (which is a particular case of our main result, but slightly different).

D. Yost \cite[Theorem 2.1]{yost1981} showed that the implication $w\lur\implies\mathsf{SC}$ does not reverse in the strong sense, namely, that every infinite-dimensional separable Banach space admits an equivalent strictly convex norm which is not $w\lur$. Of course, the analogous theorem does not hold for the implication $\lur\implies w\lur$, because of the Schur property, e.g., of the space $\ell_1$. However, it is true when assuming that the dual of the underlying space is separable; this is what our main result states:

\begin{theorem}\label{main_result}
Every infinite-dimensional Banach space with separable dual admits an equivalent $w\lur$ norm which is not $\lur$.
\end{theorem}

\begin{remark}
It is worth mentioning that the class of Banach spaces having a $w\lur$ renorming coincides with the class of Banach spaces having a $\lur$ renorming \cite[Theorem 1.11]{molto2009}. However, Theorem \ref{main_result} (and, all the more, Corollary \ref{corollary_tang}) suggests that in a large class of Banach spaces with a $w\lur$ renorming not every $w\lur$ norm is automatically $\lur$.
\end{remark}

\section{An example of a wLUR norm which is not LUR}
The norm
\begin{equation}\label{norm_c0}
\tnorm{x}=\normi{x}+\left(\sum_{n=1}^\infty2^{-n}|x(n)|^2\right)^{1/2}\quad\mbox{for }x\in c_0,
\end{equation}
where $\normi{\cdot}$ stands for the standard supremum norm, was given in \cite[p. 1]{molto2009} as an example of a strictly convex norm which is not $\lur$. Nonetheless, we shall show that this norm is $w\lur$.

\begin{lemma}\label{lemma_c0}
Suppose that $(x_n)_{n=1}^\infty\subset c_0$ is pointwise convergent to ${\alpha x}$, where ${\alpha\in[0,\infty)}$ and $x\in c_0\setminus\{0\}$. If
\[ \lim_{n\to\infty}\left(\normi{x_n+x}-\normi{x_n}\right)=\normi{x} \]
and the limit $\lim_{n\to\infty}\normi{x_n}$ exists, then $\lim_{n\to\infty}\normi{x_n}=\alpha\normi{x}$.
\end{lemma}

\begin{proof}
We shall show that the sequence $(\normi{x_n})_{n=1}^\infty$ has a subsequence which is convergent to $\alpha\normi{x}$.

Let
\[ K=\{k\colon |x(k)|=\normi{x}\}; \]
by our assumptions $K$ is a non-empty finite set. Furthermore, if $n$ and $k$ are positive integers such that ${k\notin K}$, then
\begin{align*}
|(x_n+x)(k)|-\normi{x_n} &\le|x_n(k)|+|x(k)|-\normi{x_n}\le|x(k)|\\
&\le\max\{|x(l)|\colon l\notin K\}<\normi{x}.
\end{align*}
It means that there is a $k_0\in K$ such that $|(x_n+x)(k_0)|=\normi{x_n+x}$ for infinitely many $n$. Let $(n_l)_{l=1}^\infty$ be a strictly increasing sequence of positive integers such that
\[ |(x_{n_l}+x)(k_0)|=\normi{x_{n_l}+x} \quad\mbox{for } l=1,2,\ldots. \]
Passing with $l$ to infinity we obtain
\[ (1+\alpha)|x(k_0)|=\lim_{l\to\infty}\normi{x_{n_l}+x}=\lim_{l\to\infty}\normi{x_{n_l}}+\normi{x}, \]
which completes the proof.
\end{proof}

\begin{proposition}
The norm given by \eqref{norm_c0} is $w\lur$.
\end{proposition}

\begin{proof}
Fix a sequence $(x_n)_{n=1}^\infty$ in the unit sphere of $(c_0,\tnorm{\cdot})$ and a point $x$ from this sphere such that $\lim_{n\to\infty}\tnorm{x_n+x}=2$. We shall show that each subsequence of $(x_n)_{n=1}^\infty$ has a subsequence which is weakly convergent to $x$. To this end fix an arbitrary subsequence of $(x_n)_{n=1}^\infty$ which for convenience will be still denoted by $(x_n)_{n=1}^\infty$.

Set
\[ y_n=\left(2^{-k/2}x_n(k)\right)_{k=1}^\infty\mbox{ for }n=1,2,\ldots\quad\mbox{ and }\quad y=\left(2^{-k/2}x(k)\right)_{k=1}^\infty. \]
The equality
\begin{align*}
2-\tnorm{x_n+x} &=\tnorm{x_n}+\tnorm{x}-\tnorm{x_n+x}\\
&=\normi{x_n}+\normi{x}-\normi{x_n+x}+\normt{y_n}+\normt{y}-\normt{y_n+y},
\end{align*}
where $\normt{\cdot}$ stands for the norm in $\ell_2$, implies the existence and the equality of the following limits:
\begin{equation}\label{xlinfty}
\lim_{n\to\infty}\left(\normi{x_n}+\normi{x}-\normi{x_n+x}\right)=0
\end{equation}
and
\begin{equation}\label{yl2}
\lim_{n\to\infty}\left(\normt{y_n}+\normt{y}-\normt{y_n+y}\right)=0.
\end{equation}
Passing to a further subsequence of $(x_n)_{n=1}^\infty$ (still denoted by $(x_n)_{n=1}^\infty$) we may assume that the limits $\lim_{n\to\infty}\normi{x_n}$ and $\lim_{n\to\infty}\normt{y_n}$ exist. Using the equality \eqref{yl2} we obtain
\[ \lim_{n\to\infty}(\normt{y_n}+\normt{y})^2=\lim_{n\to\infty}\normt{y_n+y}^2=\lim_{n\to\infty}(\normt{y_n}^2+2(y_n|y)+\normt{y}^2), \]
where $(\cdot|\cdot)$ stands for the real inner product. Whence
\[ \lim_{n\to\infty}(y_n|y)=\lim_{n\to\infty}\normt{y_n}\cdot\normt{y}=\alpha\normt{y}^2, \]
where $\alpha=\lim_{n\to\infty}\normt{y_n}/\normt{y}$. Thus
\begin{align*}
\lim_{n\to\infty}\normt{y_n-\alpha y}^2 &=\lim_{n\to\infty}\left(\normt{y_n}^2-2\alpha(y_n|y)+\alpha^2\normt{y}^2\right)\\
&=\alpha^2\normt{y}^2-2\alpha^2\normt{y}^2+\alpha^2\normt{y}^2=0,
\end{align*}
which means that the sequence $(y_n)_{n=1}^\infty$ is convergent (in the space $\ell_2$) to $\alpha y$. In particular, the sequence $(y_n)_{n=1}^\infty$ is pointwise convergent to $\alpha y$, therefore the sequence $(x_n)_{n=1}^\infty$ is pointwise convergent to $\alpha x$. By the equality \eqref{xlinfty} and Lemma \ref{lemma_c0}, $\lim_{n\to\infty}\normi{x_n}=\alpha\normi{x}$. Therefore
\begin{align*}
1 &=\lim_{n\to\infty}\tnorm{x_n}=\lim_{n\to\infty}\normi{x_n}+\lim_{n\to\infty}\normt{y_n}\\
&=\alpha\normi{x}+\alpha\normt{y}=\alpha\tnorm{x}=\alpha.
\end{align*}
Finally, the sequence $(x_n)_{n=1}^\infty$ is weakly convergent to $x$ as it is bounded and converges pointwise to this point.
\end{proof}

\section{The proof of the main result}
Throughout this section $X$ denotes an infinite-dimensional Banach space. We shall need a simple lemma about the weak convergence (a trivial proof will be omitted).

\begin{lemma}\label{lemma_weak}
Assume that $(x_n)_{n=1}^\infty$ is a bounded sequence in $X$, $\Gamma$ is a set and $\{x^\ast_\gamma\colon \gamma\in\Gamma \}\subset X^\ast$. If the space $\mathrm{span}\{x^\ast_\gamma\colon \gamma\in\Gamma \}$ is dense in $X^\ast$ and
\[ \lim_{n\to\infty}x^\ast_\gamma(x_n)=0\quad\mbox{for each }\gamma\in\Gamma, \]
then the sequence $(x_n)_{n=1}^\infty$ is weakly null.
\end{lemma}

\begin{proof}[Proof of Theorem \ref{main_result}]
Assume that $X^\ast$ is separable. According to the result of A. Pe{\l}czy\'nski \cite[Remark A]{pelczynski1976} there exists an $M$-basis $(e_n, e^\ast_n)_{n=1}^\infty$ of the space $X$ which is both bounded and shrinking. It means that
\[ \sup\{\norm{e_n}\!\cdot\!\norm{e_n^\ast}\colon n=1,2\ldots\}<\infty \]
and the functionals $e_n^\ast$ are linearly dense in $X^\ast$.

Without loss of generality we may assume that $\norm{e_n}=1$ for $n=1,2,\ldots$. Define a functional $\norm{\cdot}_0\colon X\to[0,\infty)$ by
\[ \norm{x}_0=\max\left\{\frac12\norm{x},\sup\limits_{n}|e_n^\ast(x)|\right\}\quad\mbox{for }x\in X. \]
One can easily see that $\norm{\cdot}_0$ is a norm on $X$ and by the boundedness of the $M$-basis $(e_n, e^\ast_n)_{n=1}^\infty$ this norm is equivalent to the original one.

Define a functional $\tnorm{\cdot}\colon X\to[0,\infty)$ by
\[ \tnorm{x}^2=\norm{x}^2_0+\sum_{n=1}^\infty4^{-n}\left|e^\ast_n(x)\right|^2\quad\mbox{for }x\in X. \]
One can easily observe that $\tnorm{\cdot}$ is an equivalent norm on $X$. We shall show that it is $w\lur$ but not $\lur$.

For the proof of the first part consider a sequence $(x_n)_{n=1}^\infty$ and a point $x$ in the unit sphere of $(X,\tnorm{\cdot})$ such that $\lim_{n\to\infty}\tnorm{x_n+x}=2$. Set
\[ y_n=\left(\norm{x_n}_0,\,2^{-1}e^\ast_1(x_n),\,2^{-2}e^\ast_2(x_n),\,\ldots\right)\quad\mbox{for }n=1,2,\ldots \]
and
\[ y=\left(\norm{x}_0,\,2^{-1}e^\ast_1(x),\,2^{-2}e^\ast_2(x),\,\ldots\right). \]
We have
\begin{align*}
\normt{y_n+y}^2 &=(\norm{x_n}_0+\norm{x}_0)^2+\sum_{m=1}^\infty4^{-m}\left|e^\ast_m(x_n+x)\right|^2\\
&\ge\norm{x_n+x}_0^2+\sum_{m=1}^\infty4^{-m}\left|e^\ast_m(x_n+x)\right|^2=\tnorm{x_n+x}^2\,\stackrel[n\to\infty]{}{\xrightarrow{\hspace*{1cm}}}4,
\end{align*}
and by the local uniform rotundity of the norm in the (Hilbert) space $\ell_2$, we obtain $\lim_{n\to\infty}\normt{y_n-y}=0$. In particular,
\[ \lim_{n\to\infty}e^\ast_m(x_n)=e^\ast_m(x)\quad\mbox{for }m=1,2,\ldots. \]
Lemma \ref{lemma_weak} and the fact that the $M$-basis $(e_n, e^\ast_n)_{n=1}^\infty$ is shrinking give the weak convergence of the sequence $(x_n)_{n=1}^\infty$ to $x$.

To see that the norm $\tnorm{\cdot}$ is not $\lur$ consider the sequence $(e_1+e_n)_{n=1}^\infty$ and the point $e_1$. One can easily verify that
\[ \lim_{n\to\infty}\tnorm{e_1+e_n}=\frac12\sqrt5=\tnorm{e_1} \]
and
\[ \lim_{n\to\infty}\tnorm{2e_1+e_n}=\sqrt5, \]
while $\tnorm{e_n}\ge1$ for $n=1,2,\ldots$.
\end{proof}

\begin{corollary}\label{corollary_tang}
Every Banach space which admits an equivalent $\lur$ norm, in particular every separable Banach space, and has an infinite-dimensional subspace with separable dual admits an equivalent $w\lur$ norm which is not $\lur$.
\end{corollary}

\begin{proof}
Suppose that $Y$ is an infinite-dimensional subspace of $X$ with separable dual. By Theorem \ref{main_result} the space $Y$ admits an equivalent $w\lur$ norm which is not $\lur$. According to Tang's Theorem \cite[Theorem 1.1]{tang1996} it extends to an equivalent $w\lur$ norm on the whole $X$. Obviously, this extension fails to be $\lur$.
\end{proof}

\begin{remark}
The statement of Tang's Theorem does not include the case of $w\lur$ norm literally, however, the theorem is also valid in this case (cf. \cite[Remark 1.1]{tang1996}). Indeed, one can easily verify that the proof works without major changes.
\end{remark}

\begin{remark}
Corollary \ref{corollary_tang} implies that every Banach space which admits an equivalent $\lur$ norm, in particular every separable Banach space, and enjoys the Schur property has no infinite-dimensional subspace with separable dual. Of course, it is not a new result, as it is well-known that every Banach space having the Schur property is $\ell_1$-saturated. However, this fact follows from Rosenthal's $\ell_1$-Theorem (cf. \cite[\S 10.2]{albiac2006}), so its proof is much less elementary than the one given in this paper.
\end{remark}

\subsection*{Acknowledgement}
This research was supported by University of Silesia Mathematics Department (Iterative Functional Equations and Real Analysis program).

\end{document}